\documentclass[12pt,a4paper,leqno,oneside]{article}
\usepackage{amsmath,amssymb,amsfonts,amsthm}
\newcommand{\BE}{{\mathbb{E}}}

\newcommand{\cmin}{c_{min}}
\newcommand{\subgraph}{H_0}
\newcommand{\subgraphs}{\mathcal{H}_0}
\newcommand{\V}{\mathcal{V}}

\newcommand{\W}{\mathcal{W}}
\newcommand{\dtv}{d_{TV}(X,P_{\lambda})}

 \newcommand{\Po}[1]{\textrm{Po}\left(#1\right)}
 
\newcommand{\E}{\mathbb{E}}

\newcommand{\Pra}[1]{\mathbb{P}\left\{#1\right\}}
\newcommand{\Gp}{\mathcal{G}\left(n,m,p\right)}

\newcommand{\Gnp}{\mathcal{G}\left(n,\hat{p}\right)}

\newcommand{\I}{\mathcal{I}}

\newcommand{\J}{\mathcal{J}}
\newcommand{\minS}{\omega}

\newcommand{\eq}{\begin{equation}}
\newcommand{\en}{\end{equation}}

\newcommand{\CC}{{\bf C}}
\newcommand{\covers}{{\mathcal{C}}}
\newcommand{\C}{{\bf C}}

\newtheorem{tw}{Theorem}
\newtheorem{lem}{Lemma}
\newtheorem{cor}{Corollary}

\newtheorem{ex}{Example}
\theoremstyle{plain}
\theoremstyle{definition}
\newtheorem{rem}{Remark}
\newtheorem{df}{Definition}

\begin{document}

\title{Poisson approximation of counts of subgraphs \\
in random intersection graphs.}

\author{Katarzyna Rybarczyk
	\thanks{
Faculty of Mathematics and Computer Science,
Adam Mickiewicz University,
Umultowska 87,
61-614 Pozna\'n, Poland,
email: kryba@amu.edu.pl}
\thanks{
Katarzyna Rybarczyk acknowledges a support by the National Science Center (NCN) grant DEC--2011/01/B/ST1/03943}
\ and Dudley Stark\thanks{
School of Mathematical Sciences, Queen Mary, University of London, London E1 4NS, United Kingdom,
email: d.s.stark@qmul.ac.uk
}
}

\maketitle

\date{}

\pagestyle{myheadings}

\begin{abstract}
Random intersection graphs are characterized by three parameters: $n$, $m$ and
$p$, where $n$ is the number of vertices, $m$ is the number of objects, and
$p$ is the probability that a given object is associated with a given vertex.
Two vertices in a random intersection graph
are adjacent if and only if they have an associated object in
common.
When $m=\lfloor n^\alpha\rfloor$ for constant $\alpha$,
we provide a condition, called {\em strictly $\alpha$-balanced}, for the
Poisson convergence of the number of induced copies of a fixed subgraph.
\end{abstract}

\section{Introduction}
The random intersection graph $\Gp$
is a probability distribution on labelled graphs.
The set of vertices of the random intersection graph $\V$ is of size
$|\V|=n$ and a second set $\W$ of size $|\W|=m$,
 called the set of {\em objects}, is used to determine the 
adjacencies in the graph.
Each vertex $v\in\V$ is associated with a set of objects $\W_v\subseteq \W$
and two vertices $v_1,v_2\in \V$ are adjacent if and only if
$\W_{v_1}\cap \W_{v_2}\neq\emptyset$.
The randomness in the graph comes by setting 
$\Pra{w\in \W_v}=p$ independently for all $w\in \W$, $v\in \V$.
The preceding description characterizes the random intersection graph
denoted by $\Gp$. The model $\Gp$ was introduced in \cite{KSS}.

Let $\subgraph$ be a given graph on $h\ge 2$ vertices and with at least 
one edge and let $K_{\V}$ denote the complete graph on the vertex set $\V$.
Let $\subgraphs$ denote the set of subgraphs of $K_{\V}$ isomorphic to 
$\subgraph$. A copy $H\in\subgraphs$ is {\em induced} in $\Gp$ if all of its edges are edges in $\Gp$ and none of its non-edges are edges in $\Gp$. 
In this paper we find conditions on $\subgraph$, $n$, $m$ and $p$
which imply that the number of induced copies of $\subgraph$ in $\Gp$ 
 has an approximately Poisson distribution.

Poisson approximation for the number
of induced copies of subgraphs has already been studied
in detail for the Erd\H os-R\'enyi model of random graphs $\Gnp$,
in which
edges between $n$ vertices
appear independently and with identical probability
$\hat{p}$;
see Chapter~6 of \cite{JLR}.
Let $e=|E(\subgraph)|$

We denote the number of induced copies of $\subgraph$ in $\Gp$
by $X=X(\subgraph)$. 
In order to facilitate our Poisson approximation of the distribution of $X$, 
we will express $X$ as a sum of indicator
random variables.
Given an integer $N$, define the set $[N]$ to be
$[N]=\{1,\ldots,N\}$
and define ${{\rm aut}(\subgraph)}$ 
to be the set of automorphisms of $\subgraph$.
The number of subgraphs of $K_\V$ isomorphic to $\subgraph$ is
\eq\nonumber
N_n:=|\subgraphs|=\binom{n}{h}\frac{h!}{|{\rm aut}(\subgraph)|}
\en
and we may index the subgraphs in $\subgraphs$ by
$$
\subgraphs=\{H_i:i\in [N_n]\}.
$$
We decompose $X$ as
\eq\label{decomp}
X=\sum_{i\in [N_n]} X_i,
\en
where $X_i$ is the indicator random variable of the event 
$\{H_i {\rm \ is\ induced \ in \ }\Gp\}$.
The intention is that the $X_i$'s should be approximately independent
and therefore $X$ should approach a Poisson distribution as $n\to\infty$
for appropriate choices of $m$ and $p$.

The total variation distance between a random variable taking nonnegative
integer values and a random variable $P_{\lambda}$ with the Poisson
distribution with parameter $\lambda$ is defined to be
$$
\dtv={\frac 1 2}\sum_{k=0}^\infty
\left|\Pra{X=k}-e^{-\lambda}\lambda^k/k!\right|.
$$
As was done in \cite{KSS,RS}, we parametrise $m=m(n)$ by
\eq\label{alphadef}
m=\lfloor n^\alpha\rfloor
\en
for some constant $\alpha>0$.
Our method of proof will be to apply Stein's method to show that
$d_{TV}(X,P_{\lambda})\to 0$ as $n\to\infty$ under suitable conditions.

Poisson approximation for the number
of induced copies of subgraphs has already been studied
in detail for the Erd\H os-R\'enyi model of random graphs,
in which
edges appear independently and with identical probability;
$\hat{p}$;
see Chapter~6 of \cite{JLR}.
Let $\subgraph$ be a graph with $e$ edges and $h$ vertices.
Given $S\subseteq V(\subgraph)$, we define $E(S)$ to be the
set of edges of $\subgraph$ having both vertices in $S$.
A graph $\subgraph$ is called {\em strictly balanced} if
\eq\label{cond}
\max_{\emptyset\varsubsetneq S\varsubsetneq V(\subgraph)}
\frac{|E(S)|}{|S|}
<\frac{e}{h}.
\en
Let $W$ denote the number of not necessarily induced copies of 
$\subgraph$ in $\Gnp$
and let
$$
\lambda=\BE(W)={n\choose h}\frac{h!}{|{\rm aut}(\subgraph)|} \hat{p}^e.
$$
Define $\kappa=\kappa(\subgraph)$ by
\eq\nonumber
\kappa=\min_{\emptyset\varsubsetneq S\varsubsetneq V(\subgraph)}
|E(S)|\left(\frac{|S|}{|E(S)|}-\frac{h}{e}\right).
\en
Bollob\'as \cite{Bol}
shows Poisson convergence of $W$ through the method of moments.
Theorem~5.B of \cite{BHJ} gives the bound
\eq\label{Barbour}
d_{\rm TV}(W,P_{\lambda})=
\left\{
\begin{array}{l l}
O(1)\lambda^{1-1/e}n^{-\kappa}&{\rm if \ }\lambda\geq 1;\\
O(1)\lambda\,n^{-\kappa}&{\rm if \ }\lambda<1.
\end{array}
\right.
\en
When $\hat{p}$ is such that $\lambda\to\lambda_0$ for a constant $\lambda_0$,
then (\ref{Barbour}) implies that the distribution of
$W$ converges in total variation distance
to a Poisson$(\lambda_0)$ distrtibution. That is not the case for subgraphs
which are not strictly balanced.

The only subgraphs for which the asymptotic distribution of $X(\subgraph)$ has
been determined for $\Gp$ are $\subgraph=K_h$, the complete graphs on $h$ vertices,
in \cite{RS}, in which $X(K_h)$ was shown to have a limiting Poisson
distribution at the threshold for the appearance of $K_h$.
Theorem~\ref{prev-result} from \cite{RS} for complete graphs
is the kind of result we have in mind to extend to general $\subgraph$.
For a constant $c>0$, we parametrise $p=p(n)$ by
\begin{equation}\label{pdef}
p(n)\sim
\begin{cases}
\begin{tabular}{ll}
$c\: n^{-1}m^{-\frac{1}{h}}$\ &\ for $0<\alpha< \frac{2h}{h-1}$;\\
$c\:  n^{-\frac{h+1}{h-1}}$ \ &\ for $\alpha=\frac{2h}{h-1}$;\\
$c\: n^{-\frac{1}{h-1}}m^{-\frac{1}{2}}$\ &\ for $\alpha> \frac{2h}{h-1}$.
\end{tabular}
\end{cases}
\end{equation}
We focus on asymptotic values thus we will use standard Landau notation $O(\cdot), o(\cdot)$, $\Omega(\cdot)$, $\sim$, and $\asymp$ 
as in \cite{JLR}.
The following theorem was proved in \cite{RS}.
\begin{tw}\label{prev-result}
Let $\Gp$ be a random intersection graph defined
with $m$ and $p$ given in terms of $n$ by
(\ref{alphadef}) and (\ref{pdef})
and let $h\geq 3$ be a fixed integer.
Let $X_n$ be the random variable counting the number of instances of
$K_h$ in $\Gp$.
\begin{enumerate}
\item[(i)] If
$\alpha < \frac{2h}{h-1}$, then $\lambda_n=\mathbb{E} X_n\sim c^h/h!$ and
$$d_{TV}(X_n,P_{\lambda_n})=O\left(n^{-\frac{\alpha}{h}}\right);$$
\item[(ii)] If $\alpha = \frac{2h}{h-1}$, then
$\lambda_n=\mathbb{E} X_n\sim\left(c^h+c^{h(h-1)}\right)/h!$ and
$$d_{TV}(X_n,P_{\lambda_n})=O\left(n^{-\frac{2}{h-1}}\right);$$
\item[(iii)] If $\alpha>\frac{2h}{h-1}$, then for
$\lambda_n=\mathbb{E} X_n\sim c^{h(h-1)}/h!$ and
$$d_{TV}(X_n,P_{\lambda_n})=O\left(n^{\left(h-\frac{\alpha(h-1)}{2}-\frac{2}{h-1}\right)}+
n^{-1}\right).$$
\end{enumerate}
\end{tw}

The different cases in Theorem~\ref{prev-result}
arise from the ways copies of $K_h$ can appear in $\Gp$.
The main ways are either
that a single object is responsible for the existence of
every edge in the clique or
that each edge appears because of a different object associated with it.
There are other
ways in which copies of $K_h$ can appear, but asymptotically they are
unimportant.
It may happen that one of the two main ways is dominant or that both 
ways contribute.
In case $(i)$, the first way dominates; in case $(iii)$,
 the second way dominates; and in case $(ii)$,
both cases contribute.

The ways copies can appear in $K_h$ are described in \cite{KSS} by using
the notion of clique covers, to be defined in Section~2.
Using clique covers, \cite{KSS} proves a theorem
for the model $\Gp$ showing how to compute the threshold
for the appearance of subgraphs of $\Gp$.
We will extend the idea behind the
definition of strictly balanced graphs to 
clique covers and thereby derive a Poisson approximation
result for subgraph counts in $\Gp$.
In Section~2 we define strictly balance clique covers and state
Theorem~\ref{main}, our main result.

\section{Strictly balanced clique covers}

The idea of categorising the various ways a subgraph can appear
was formalised in \cite{KSS} through the notion of clique covers.
The following definitions are taken from \cite{KSS}.
Given a fixed subgraph $\subgraph$ of $K_\V$,
define $V(\subgraph)$ and $E(\subgraph)$ to be the vertex and edge
sets of $\subgraph$, respectively.
\begin{df}
A {\em clique cover} 
$\CC=\{C_1, \ldots, C_{t}\}$ of $\subgraph$
is a set of non-empty  subsets of
$V(\subgraph)$  such that
\begin{enumerate}
\item[$(i)$] each $C_i\in\CC$ induces a clique in $\subgraph$; 
\item[$(ii)$] for any
$\{v_1,v_2\}\in E(\subgraph)$ there exists $C_i\in\CC$ such that
$v_1,v_2\in C$.
\end{enumerate}
\end{df}
\noindent If in addition 
\begin{itemize}
\item[$(iii)$] $|C_i|\ge 2$ for all $C_i\in \CC$
\end{itemize}
we call a clique cover {\em proper}.

We call $t$ the {\em size} of the clique cover.
By the definition of clique cover, the cliques induced in $K_{\V}$ by the
sets in $\CC$ cover all the edges of $\subgraph$ and no other edges.
There are clearly only a finite number of clique covers of
$\subgraph$.
We denote the finite set of {\em proper} clique covers  of $\subgraph$ 
by $\covers(\subgraph)$.


If $w\in\W_v$, then we say that $w$ has been {\em chosen} by $v$.
In $\Gp$, the set of vertices which have chosen
a particular object $w\in \W$ always form a clique in $\Gp$
and, therefore, the set of edges in $\Gp$ is the union of the $m$ edge-sets
of the cliques generated by the elements of $\W$.
\begin{df}\label{induced}
We say that $\subgraph \subseteq\Gp$
is {\em induced} by a clique cover
$\C=\{C_1, \ldots, C_{t}\}$ of $\subgraph$
if there is a family of disjoint non-empty subsets
$\{W_1, \ldots, W_{t}\}$ of $\W$, such that, 
\begin{enumerate}
\item[$(i)$]  for all $i\in [t]$, each element of $W_i$
is an object chosen by all the vertices of $C_i$ and
no other vertices from $V(\subgraph)$; 
\item[$(ii)$]  each $w\in \W\setminus \bigcup_{i=1}^t W_i$
is chosen by at most one vertex from $V(\subgraph)$.
\end{enumerate}
\end{df}
Clearly, if $\subgraph$ is an induced subgraph of $\Gp$,
then it is induced by exactly one clique cover from $\covers(\subgraph)$ in $\Gp$.

Letting $\C=\{C_1,\ldots,C_t\}$ be any clique cover of $\subgraph$ (not necessarily in $\covers(\subgraph)$),
we denote:
$|\C|=t$ and $\sum\C=\sum_{i=1}^{t}|C_i|$.
We define the clique cover derived from $\C$ containing
only cliques of size at least two by
$\C^\prime=\{C_i: |C_i|\geq 2, \ i\in[t]\}$.
Given $\emptyset\varsubsetneq S\varsubsetneq V(H)$,
we define two different types of {\em restricted clique covers},
which are multisets defined  by
$$
\C[S]:=\{C_i\cap S: |C_i\cap S|\geq 1, \ i\in[t]\}
$$
and
$$
\C^\prime[S]:=\{C_i\cap S: |C_i\cap S|\geq 2, \ i\in[t]\}.
$$

Let $\subgraph[S]$ be the subgraph of $\subgraph$ induced by $S$.
We say that restricted clique cover
$\C[S]$ (respectively $\C^\prime[S]$) induces
$\subgraph[S]$ if Definition~\ref{induced} is satisfied with
$\subgraph$ replaced by $\subgraph[S]$ and the set
$\C$ replaced by the multiset $\C[S]$
(respectively $\C^\prime[S]$).
If $\C$ induces $\subgraph$, then $\C[S]$ and $\C'[S]$ induce $\subgraph[S]$.
The subset $S$ of vertices  plays a similar role here as it does in 
definition (\ref{cond}) of strictly balanced subgraphs. 
The restricted clique covers are defined to be multisets because,
if $C_i\cap S=C_j\cap S$ for $i\neq j$,
and if $\subgraph$ is induced by $\C$, then vertices from 
$C_i\cap S=C_j\cap S$ must still
 choose objects from disjoint non-empty subsets
$W_i$, $W_j$.
Moreover, if $|C_i\cap S|=1$, then there must still be an object chosen by the single vertex in 
$C_i\cap S$ in order for $\C$ to induce $\subgraph$. This explains
why restricted cliques of size 1 are included in the definition of
$\C[S]$. It is shown in 
\cite{KSS}
that the order of the expected
number of copies of $\subgraph[S]$ induced by $\C[S]$ 
does not depend on the number of restricted cliques of size 1 when
$mp$ is bounded below, because in that case
the probability that
at least one object will choose any given vertex is bounded below. Thus,
$\C$ is important when $mp=o(1)$ and $\C^\prime$ is important
when $mp$ is bounded below.

Define the sizes of the multisets $\C[S]$, $\C^\prime[S]$ by
$$
\sum\C[S]=\sum_{\stackrel{i\in [t]}{|C_i\cap S|\geq 1}}
|C_i\cap S|
$$
and
$$
\sum\C^\prime[S]=\sum_{\stackrel{i\in [t]}{|C_i\cap S|\geq 2}}
|C_i\cap S|.
$$
Let $X(\subgraph,\C,S)$ denote the number of copies of $\subgraph[S]$ induced
by $\C$ and $\C^\prime$.
It is shown in \cite{KSS}
that, assuming $mp^2=o(1)$, 
$$
\BE(X(\subgraph,\C,S))\asymp 
\psi(\subgraph,\C,S):=
\min\left\{n^{|S|+\alpha |\C[S]|} p^{\sum\C[S]},n^{|S|+\alpha |\C^\prime[S]|} p^{\sum\C^\prime[S]}\right\}.
$$

We are interested in $p=p(n)$ such that 
$\BE(X(\subgraph,\C,S))\asymp 1$.
For this purpose define
\begin{equation}\label{EqEta2}
\eta_2(\subgraph,\C,S):=
\begin{cases}
\frac{|S|+\alpha|\CC[S]|}{\sum \CC[S]}
&\text{ if either }\alpha < \frac{|S|}{\sum \CC [S]-|\CC[S]|}\text{ or }\sum\CC [S]=|\CC[S]|;\\
\frac{|S|+\alpha |\CC'[S]|}{\sum \CC'[S]}
&\text{ otherwise,}
\end{cases}
\end{equation}
so that $\psi(\subgraph,\C,S) = 1$ when $p=n^{-\eta_2(\subgraph,\C,S)}$, and define
\begin{equation}\label{EqEta1}
\eta_1(\subgraph,\C):=
\min_{\emptyset\varsubsetneq S\subseteq V(\subgraph)}
\eta_2(\subgraph,\C,S).
\end{equation}
The previous comments indicate that $p=n^{-\eta_1(\subgraph,\C)}$ should be
a threshold for $\C$ inducing copies of $\subgraph$. In other words, if
$p\ll n^{-\eta_1(\subgraph,\C)}$, then $\C$ will induce 0 copies of $\subgraph$
almost surely (as $n\to\infty$), but, if $p\asymp n^{-\eta_1(\subgraph,\C)}$,
then $\C$ will induce some copies of $\subgraph$ with positive probability.
Thus, if we define
\begin{equation}\label{EqEta0}
\eta_0=\eta_0(\subgraph):=\max_{\CC\in\covers(\subgraph)} \eta_1(\subgraph,\C),
\end{equation}
then it is natural to expect that under suitable conditions
$p=n^{-\eta_0(\subgraph)}$ should be the threshold for the appearance of 
$\subgraph$ in $\Gp$. The main result of \cite{KSS}, which applies for
general $m$, not just $m$ of the form (\ref{alphadef}), shows
that $n^{-\eta_0(\subgraph)}$ actually
is the threshold for the appearance of $\subgraph$ in $\Gp$.

We now proceed with our results for Poisson approximation.
We call a clique cover $\C\in\covers(\subgraph)$ 
{\em strictly} $\alpha$-{\em balanced} if
$\eta_2(\subgraph,\C,S)>\eta_2(\subgraph,\C,V(\subgraph))$
for all $\emptyset\varsubsetneq S\varsubsetneq V(\subgraph)$.
The clique covers which will induce copies of $\subgraph$ at threshold
$p=n^{-\eta_0(\subgraph)}$ are those in the set
\begin{equation}\label{EqC0}
\covers_0=\covers_0(\subgraph):=\{\CC\in\covers(\subgraph): \eta_1(\subgraph,\CC)=\eta_0\}.
\end{equation}
We call $\subgraph$ 
strictly $\alpha$-balanced if 
all $\C\in\covers_0$ are strictly $\alpha$-balanced.

Our main result, Theorem~\ref{main}, 
gives new conditions
for the Poisson convergence of $X(\subgraph)$ at the threshold of
appearance of $\subgraph$. Theorem~\ref{main}  can be applied,
for example, when $\subgraph$
is a $K_h$, a $C_h$ (cycle), or triangle-free.
\begin{tw}\label{main}
Let $\subgraph$ be a given graph, $m={\lfloor n^\alpha\rfloor}$, for $\alpha>0$,  $\eta_0$ be given by \eqref{EqEta2}, \eqref{EqEta1} and \eqref{EqEta0}, and $\covers_0$ be defined by \eqref{EqC0}.
Suppose that  
$p=cn^{-\eta_0}$ for a constant $c>0$
and that $mp^2=o(1)$.
If $\subgraph$ is strictly $\alpha$-balanced, then 
$$
d_{TV}(X,P_{\lambda_0})=o(1)
$$
for 
$$
\lambda_0=
\frac{1}{|{\rm aut}(\subgraph)|} 
\sum_{\C\in\covers_0(\subgraph)}c^{\sum\C},
$$
where $X$ is the number of induced copies of $\subgraph$ in $\Gp$.
\end{tw}
\noindent 
The meaning of $\lambda_0$ is that
it is the limit of the number of copies of $\subgraph$ induced by
clique covers in $\covers_0$.
We do not know if it is possible that $\BE(X)\not\to\lambda_0$.

Some of our results might be obtained using reasoning from~\cite{KR} to the random bipartite graph with bipartition $(\V,\W)$,
however, due to double counting of ways subgraphs
can be generated, it would seem to require great effort. 
Moreover, the treatment of the problem in this article might be 
useful in the developing arguments for other problems concerning subgraph 
counts in $\Gp$.

To illustrate Theorem~\ref{main}, we apply it to a
triangle-free $\subgraph$ and compare the bounds obtained
to related results for $\Gnp$ with $\hat{p}= mp^2$.
We obtain the following bound:
\begin{cor}\label{CorTrianglefree}
Let $\subgraph$ be strictly balanced triangle-free graph on $h$ vertices and with $e$ edges. If $\alpha>h/e$ then for $p=cn^{\frac{h+\alpha e}{2e}}$ we have  
\eq\label{strictly}
d_{\rm TV}\left(X,P_{\lambda_0}\right)=o(1),
\en
for $\lambda_0=\frac{c^{2e}}{|{\rm aut}(\subgraph)|}$,
where $X$ is the number of induced copies of $\subgraph$.
\end{cor}
\noindent For the proof of Corollary~\ref{CorTrianglefree} see the Appendix.

The bound (\ref{strictly}) gives Poisson convergence 
analagous to 
(\ref{Barbour}) for $\Gnp$ with $\hat{p}=mp^2$, but only applies for
$\alpha>h/e$. It is typical for results in $\Gp$ to show behaviour
similar to that in $\Gnp$ for $\alpha$ large enough or small $p$. A result
of \cite{FSC} and \cite{Kasia}  show equivalence in total variation distance
between $\Gp$ and $G(n,\hat{p})$ for
$\hat{p}$ chosen appropriately when $\alpha>6$ or $p=o(n^{-1}m^{-1/3})$.
Note that the results of \cite{Kasia} concerning $\alpha>3$  do not apply because
the event of having an induced subgraph is not an increasing event.

For the case $\alpha\le h/e$,
recall that for a triangle-free graph $\subgraph$, the set $\covers(\subgraph)$ contains only one clique cover $\C$ consisting of $2$--element sets. Moreover, for a strictly balanced $\subgraph$, if $\alpha\le h/e$, then for every $S \subseteq V(\subgraph)$ such that $\subgraph[S]$ has at least one edge
$$
\frac{|S|}{\sum \C[S]-|\C|}=\frac{|S|}{E(S)} \ge \frac{h}{e} \ge \alpha.
$$
Thus $\C[S]$ always contributes to $\eta_2$ and, in general, we should not expect similarity with $\Gnp$. However the following special cases show that, depending on the graph the value $h/e$ is not always critical. 
Theorem~\ref{main}, Corollary~\ref{CorTrianglefree}, and simple calculations lead to the following. 
\begin{ex}
Let $C_t$ be a cycle on $t\ge 4$ vertices and $p=c n^{-\frac{1}{2}-\frac{\alpha}{2}}$. Then for any $\alpha>0$ and $\lambda_0=\frac{c^{2t}}{2t}$ we have ${\rm d_{TV}}(X(C_k),Po(\lambda_0))=o(1)$
\end{ex}

\begin{ex} Let
$K_{k,t}$ be a bipartite graph with $t>k$. Then for $p=cn^{-\frac{k+t}{2kt}-\frac{\alpha}{2}}$ and $\lambda_0=\frac{c^{2kt}}{k!t!}$
 we have ${\rm d_{TV}}(X(K_{k,t}),Po(\lambda_0))=o(1)$ for $\alpha > \frac{t-k}{tk}$. 
\end{ex}

Results on the asymptotic probabilities of clique covers
which give the asymptotics of $\E(X)$
are derived in Section~3. 
Section~4 contains results
regarding the second moment of $X$.
In Section~5, we use Stein's method to prove Theorem~\ref{main}.

\section{Asymptotic probabilities}
Lemma~\ref{LematSzacowania}, which shows that
the numbers of objects inducing cliques are asymptotically independent,
 is similar to Lemma~1 of \cite{KSS}.
One difference between Lemma~\ref{LematSzacowania} and Lemma~1 of \cite{KSS} 
is that our conditions on $A_0$ are stated explicitly and another
is that we present the proof in full detail.
In Lemma~\ref{LematSzacowania} and Lemma~\ref{cliqueasymp} we
allow the possiblity that $|C_i|=1$. Thus $\C[S]$ and $\C^\prime[S]$
are written as
$\C$ in those lemmas to simplify notation.
\begin{lem}\label{LematSzacowania}
Given a clique cover $\C=\{C_1,C_2,\ldots,C_r\}$
of $\subgraph$,
let $\cmin=\min_{1\leq i\leq r}|C_i|$. Moreover, let $N_i$ be the number of objects, which have been chosen by all vertices from $C_i$ and no vertex from 
$V(\subgraph)\setminus C_i$ in $\Gp$ and let $\tilde{N}_i$ have Poisson distribution $\Po{mp^{|C_i|}}$. If $mp^{\cmin+1}=o(1)$ then
$$
\Pra{\bigcap_{i=1}^r \{N_i = a_i\}}
\sim
\prod_{i=1}^r\Pra{\tilde{N}_i = a_i}
$$
uniformly over all $a_i\le A_0$ for any 
$A_0=A_0(n)$ satisfying $A_0=o(\sqrt{m})$ and $A_0=o(p^{-\cmin})$.
\end{lem}
\begin{proof}
Define 
\begin{equation}\label{pidef}
p_i=p^{|C_i|}(1-p)^{h-|C_i|},\quad 1\leq i\leq r,
\end{equation}
which is the probability that a given object is chosen by all 
vertices in $C_i$ and no other vertices in $V(\subgraph)$.
Note that $p_i\leq p^{\cmin}$.
Let $p_0=1-\sum_{i=1}^rp_i$. Define $a_0=m-\sum_{i=1}^ra_i$.
If $1\le a_i\le A_0$ for all $1\leq i\leq r$, then
\begin{align*}
&\Pra{\bigcap_{i=1}^r \{N_i = a_i\}}
=
\binom{m}{a_0,\ldots,a_t}
p_0^{a_0}\prod_{i=1}^rp_i^{a_i}\\
&=
\frac{m!}{(m-\sum_{i=1}^r a_i)!}
\left(1-\sum_{i=1}^rp_i\right)^{m-\sum_{i=1}^r a_i}
\prod_{i=1}^r\frac{p_i^{a_i}}{a_i!}
\\
&=
m^{\sum_{i=1}^ra_i}
\exp\left(O\left(\frac{(rA_0)^2}{m}\right)\right)
\exp\left(-m\sum_{i=1}^rp_i+O\left(r^2A_0p^{\cmin}+mr^2p^{2\cmin}\right)\right)
\prod_{i=1}^r\frac{p_i^{a_i}}{a_i!}
\\
&\sim
\exp\left(-m\sum_{i=1}^rp_i\right)
\prod_{i=1}^r\frac{(mp_i)^{a_i}}{a_i!}
\\
&=
\exp\left(-\sum_{i=1}^rmp^{|C_i|}+O\left(mp^{\cmin+1}\right)\right)
\prod_{i=1}^r\frac{(mp^{|C_i|})^{a_i}}{a_i!}(1-p_i)^{a_i(h-|C_i|)}
\\
&\sim
\prod_{i=1}^re^{-mp^{|C_i|}}\frac{(mp^{|C_i|})^{a_i}}{a_i!}\exp\left(O(hA_0p^{\cmin})\right)
\\
&\sim
\prod_{i=1}^re^{-mp^{|C_i|}}\frac{(mp^{|C_i|})^{a_i}}{a_i!}
=
\prod_{i=1}^r\Pra{\tilde{N}_i = a_i}.
\end{align*}
\end{proof}

We let $\pi(\subgraph,\CC)$
denote the probability that $\subgraph$ is induced by clique cover
$\CC$. Using this definition, the fact that $H_i\in \subgraphs$ can be induced by
at most one clique cover from $\covers(\subgraph)$, and symmetry, we calculate
that, for each $i\in[N_n]$, the expectation of the random
variable $X_i$ appearing in (\ref{decomp}) equals
\eq\nonumber
\E(X_i)=\sum_{\CC\in\covers(H_i)} \pi(H_i,\CC)
=\sum_{\CC\in\covers(\subgraph)} \pi(\subgraph,\CC).
\en

In (\ref{state1}) below, we give asymptotics for $\pi(\subgraph,\C)$.
The similar results
on pages 138--139 of \cite{KSS}
only provide asymptotics when $mp=o(1)$ and $\C$ has no cliques of size 1.
It is indeed the case that $mp\to\infty$ in the examples 
after Corollary~\ref{CorTrianglefree} when $\alpha>1$.
The order estimate (\ref{state2}) was obtained in \cite{KSS}.
Recall that if $mp^2=o(n^{-2})$ then with probability 
tending to one as $n\to\infty$ $\Gp$ is edgeless and if $\ln n=o(mp^2)$, then  with probability tending to one $\Gp$ is the complete graph (see \cite{FSC}).  
\begin{lem}\label{cliqueasymp}
Given a clique cover $\C=\{C_1,C_2,\ldots,C_r\}$
of $\subgraph$, let 
$\I_1=\{1\leq i\leq r:|C_i|=1\}$ and
$\I_2=\{1\leq i\leq r:|C_i|\geq 2\}$.
Then, 
\begin{equation}\label{state1}
\pi(\subgraph,\C)\sim 
(1-e^{-mp})^{|\I_1|}\prod_{i\in\I_2}mp^{|C_i|},
\quad {\rm for \ }\Omega(n^{-1})=mp^2=o(1).
\end{equation}
It follows that
\begin{equation}\label{state2}
\pi(\subgraph,\C)\asymp
\min\left\{m^{|\C|}p^{\sum\C}, m^{|\C^\prime|}p^{\sum\C^\prime} \right\}.
\end{equation}
\end{lem}
\begin{proof} 
Let $(C_{r+1},C_{r+2},\dots,C_t)$ list the subsets of $V(\subgraph)$
not in $\C$ of cardinality of at least $2$.
For each $1\leq i\leq t$, let $N_i$ be the number of objects which have been chosen by every vertex in $C_i$, 
and by no vertex in $V(H)\setminus C_i$.
Let $\tilde{N_i}$ be a random variable with the Poisson distribution
$
\Po{mp^{|C_i|}}.
$
The distribution of $N_i$ is ${\rm Binomial}(m,p_i)$, where
$p_i$ is defined by (\ref{pidef}). Chernoff's bound (see for example Theorem~2.1 \cite{JLR})
implies that for $A$ large enough  $\Pra{N_i\ge A_0} = o(\pi(\subgraph,\C))$ and 
$\Pra{\tilde{N_i}\ge A_0} = o(\pi(\subgraph,\C))$, where 
$
A_0(n)= A\max\{mp,\log n\}.
$
Therefore,
\begin{eqnarray*}
\pi(\subgraph,\C)
&=&\Pra{\bigcap_{i=1}^r\{N_i\ge 1\}\cap \bigcap_{j=r+1}^t\{N_j = 0\}}
\\
&=&
\sum_{
a_i\ge 1 \ {\rm for \ }1\leq i\leq r
}
\Pra{\bigcap_{i=1}^r\{N_i=a_i\}\cap \bigcap_{j=r+1}^t\{N_j = 0\}}
\\
&=&
\sum_{
1\le a_i\le A_0\text{ for }1\leq i\leq r
}
\Pra{\bigcap_{i=1}^r\{N_i=a_i\}\cap \bigcap_{j=r+1}^t\{N_j = 0\}}
+o(\pi(\subgraph,\C)),
\end{eqnarray*}

We know that $\cmin\ge 1$, where $\cmin$ is defined as in Lemma~\ref{LematSzacowania},
and $mp^{\cmin+1}=O(mp^2)=o(1)$.
Moreover, $A_0=o(\sqrt{m})$ and $A_0p=o(1)$.
By Lemma~\ref{LematSzacowania}, we now have
\begin{eqnarray*}
\pi(\subgraph,\C)
&\sim&
\sum_{
1\le a_i\le A_0\text{ for }1\leq i\leq r
}
\,
\prod_{i=1}^r\Pra{\tilde{N}_i = a_i}
\prod_{j=r+1}^t\Pra{\tilde{N}_j = 0}
\\
&\sim&
\prod_{i=1}^r\Pra{1\le \tilde{N}_i \le A_0}
\prod_{j=r+1}^t\Pra{\tilde{N}_j = 0}\\
&\sim&
\prod_{i=1}^r\Pra{\tilde{N}_i \ge 1}
\prod_{j=r+1}^t\Pra{\tilde{N}_j = 0}\\
&=&
\prod_{i=1}^r(1-\exp(-mp^{|C_i|}))\prod_{j=r+1}^t\exp(-mp^{|C_j|})\\
&\sim&
\prod_{i=1}^r(1-\exp(-mp^{|C_i|}))\\
&\sim&
\prod_{i\in\I_1}\left(1-e^{-mp}\right))
\prod_{i\in\I_2}mp^{|C_i|},
\end{eqnarray*}
proving (\ref{state1}).

To show (\ref{state2}), we observe that
if $mp\geq 1$ then $1-e^{-mp}\asymp 1$
and if $mp<1$ then 
$1-e^{-mp}\asymp mp$. Thus
$1-e^{-mp} \asymp \min\{mp,1\}$.
\end{proof}
\begin{rem}
Similar techniques lead to the following equation which might be useful 
in the study of  subgraph counts above the threshold for subgraph appearance. 
For any clique cover  $\C=\{C_1,C_2,\ldots,C_r\}$ 
of $\subgraph$, if $\Omega(1)= mp^2=O(\log n)$ then
$$
\pi(\subgraph,\C)\sim (1-e^{-mp^2})^{|\I_2|}(e^{-mp^2})^{\binom{h}{2}-|\I_2|}\prod_{i\in\I_3}mp^{|C_i|},
$$
 where
$\I_2=\{1\leq i\leq r:|C_i|=2\}$
 and
$\I_3=\{1\leq i\leq r:|C_i|\geq 3\}$.
\end{rem}

\section{Asymptotic moments}
We now define
\eq\label{minSdef}
\minS(\subgraph,\C)=\min_{\emptyset\varsubsetneq S\varsubsetneq V(\subgraph)}
\psi(\subgraph,\C,S).
\en
Lemma~\ref{moment} is a second moment calculation for copies of $\subgraph$
induced by given clique covers $\C_1$, $\C_2$. A 
calculation resulting in a special case of
Lemma~\ref{moment} was made in \cite{KSS}.
\begin{lem}\label{moment}
Suppose that $mp^2=o(1)$.
Let $G_1$ and $G_2$ be two, not necessarily isomorphic, subgraphs of $K_{\V}$ on
$g_1:=|V(G_1)|$ and $g_2:=|V(G_2)|$ labelled vertices, respectively, which intersect on
$\ell:=|V(G_1)\cap V(G_2)|$ vertices and such that $G_1\cap G_2$
is an induced subgraph of both of $G_1$ and $G_2$.
Let $\C_1$ and
let $\C_2$ be proper clique covers of $G_1$ and $G_2$, respectively.
Define
$X(G_i,\C_i)$  to be the
indicator random variable of the event 
that $G_i$ is induced by a clique cover $\C_i$ in $\Gp$.
Then, for $\Gp$,
\eq\label{interm}
\E(X(G_1,\C_1) X(G_2,\C_2))=
O(1)\E(X(G_1,\C_1)\E X(G_2,\C_2)) 
\frac{n^\ell}{\minS(G_2,\C_2)}.
\en
\end{lem}
\begin{proof}
Suppose that $\C_1=\{C_{1,1},C_{1,2},\ldots,C_{1,r}\}$ and
$\C_2=\{C_{2,1},C_{2,2},\ldots,C_{2,s}\}$.
Let $\C_1+\C_2$ denote
the set of clique covers on $V(G_1\cup G_2)$
such that $\C\in\C_1+\C_2$ implies $\C[V(G_1)]=\C_1$ and
$\C[V(G_2)]=\C_2$.
If $\C\in \C_1+\C_2$ and $C\in \C$, then $C$ must be
one of the following three forms:
\begin{enumerate}
\item[(i)] $C=C_{1,i}$ for some $1\le i\le r$ but $\forall_{1\le j\le s}(C\neq C_{2,j})$; 
\item[(ii)] $C=C_{2,j}$ for some $1\le j\le s$ but $\forall_{1\le i\le r}(C\neq C_{1,j})$;
\item[(iii)]  $C=C_{1,i}\cup C_{2,j}$ for some $1\le i\le r$, $1\le j\le s$ 
(including the case $C=C_{1,i} = C_{2,j}$).
\end{enumerate}

Given $\C\in \C_1+\C_2$, let $\J_1=\J_{1}(\C)\subseteq [r]$, 
$\J_2=\J_{2}(\C)\subseteq [s]$ and $\J_3=\J_{3}(\C)\subseteq [r]\times [s]$ 
be the sets of indices for which  (i), (ii) and (iii),
respectively, are true for some $C\in\C$.
Define
$$
\J_4=\{i\in [r]:\exists j\in [s] {\rm \ such \ that \ } (i,j)\in\J_3\}
$$
and
$$
\J_5=\{j\in [s]:\exists i\in [r] {\rm \ such \ that \ } (i,j)\in\J_3\}
$$
As $\C[V(G_1)]=\C_1$ and
$\C[V(G_2)]=\C_2$, it must be true that
\eq\label{Jcond}
\J_1\cup\J_4=[r] {\rm \ and \ } \J_2\cup\J_5=[s].
\en
Note that the clique covers in $\C_1+\C_2$ are proper because
$\C_1$ and $\C_2$ are proper.
Therefore, if $G_1$ is induced by $\C_1$ on $V(G_1)$
and $G_2$ is induced by $\C_2$ on $V(G_2)$, then
$G_1\cup G_2$ is induced by a unique element of $\C_1+\C_2$ and
\eq\nonumber
\E X(G_1,\C_1)X(G_2,\C_2)=\sum_{\C\in \C_1+ \C_2}\E X(G_\C,\C).
\en
Moreover, for any $\C\in \C_1+\C_2$, Lemma~\ref{cliqueasymp} implies
\begin{eqnarray}\label{Isim}
\E X(G_\C,\C)
&\sim& 
\prod_{\stackrel{i\in \J_1}{|C_{1,i}|>1}}mp^{|C_{1,i}|}
\prod_{\stackrel{i\in \J_1}{|C_{1,i}|=1}}(1-e^{-mp})
\prod_{\stackrel{i\in \J_1}{|C_{2,i}|>1}}mp^{|C_{2,i}|} \prod_{\stackrel{i\in \J_1}{|C_{2,i}|=1}}(1-e^{-mp})\\
&&\times\prod_{\stackrel{(k,l)\in \J_3}{|C_{1,k}\cup C_{2,l}|>1}}mp^{|C_{1,k}\cup C_{2,l}|}
\prod_{\stackrel{(k,l)\in \J_3}{|C_{1,k}\cup C_{2,l}|=1}}
(1-e^{-mp}).\nonumber
\end{eqnarray}
We will analyze (\ref{Isim})
separately for the cases $mp\leq 1$ and $mp>1$.

When $mp\leq 1$, we have $1-e^{-mp}\leq mp$, and so
$$
\E X(G_\C,\C)=
O(1)\prod_{i\in \J_1}mp^{|C_{1,i}|}
\prod_{j\in \J_2}mp^{|C_{2,j}|}
\prod_{(k,l)\in \J_3}mp^{|C_{1,k}\cup C_{2,l}|}.
$$
Now, (\ref{Jcond}) and $mp\leq e(1-e^{-mp})$ give
\begin{eqnarray*}
\E X(G_\C,\C)
&=&
O(1)\prod_{i\in\J_1}mp^{|C_{1,i}|}
\prod_{j\in\J_2}mp^{|C_{2,j}|}
\prod_{(k,l)\in \J_3}\frac{mp^{|C_{1,k}|}mp^{|C_{2,l}|}}{mp^{|C_{1,k}\cap C_{2,l}|}}\\
&=&
O(1)\prod_{i\in [r]}mp^{|C_{1,i}|}
\prod_{j\in [t]}mp^{|C_{2,j}|} 
\prod_{(k,l)\in \J_3}\frac{1}{mp^{|C_{1,k}\cap C_{2,l}|}}.
\end{eqnarray*}

It must be the case that $(C_{1,k}\cup C_{2,l})\cap V(G_1)=C_{1,k}$ and that
$(C_{1,k}\cup C_{2,l})\cap V(G_2)=C_{2,l}$, which implies that
$C_{2,l}\cap (V(G_1)\cap V(G_2)) = C_{1,k}\cap C_{2,l}$.
Therefore, by definition (\ref{minSdef}),
$$
\prod_{(k,l)\in \J_3}mp^{|C_{1,k}\cap C_{2,l}|}\ge
\prod_{C\in \C_2[V(G_1)\cap V(G_2)]}mp^{|C|}
\ge \frac{\minS(G_2,\C_2)}{n^\ell},
$$
proving (\ref{interm}).

In case $mp>1$, by \eqref{Isim} we have
$$
\E X(G_\C,\C)=
O(1)\prod_{\stackrel{i\in \J_1}{|C_{1,i}|>1}}mp^{|C_{1,i}|}
\prod_{\stackrel{i\in \J_1}{|C_{2,i}|>1}}mp^{|C_{2,i}|} 
\prod_{\stackrel{(k,l)\in \J_3}{|C_{1,k}\cup C_{2,l}|>1}}mp^{|C_{1,k}\cup C_{2,l}|}.
$$
We now bound terms of the form $mp^{|C_{1,k}\cup C_{2,l}|}$ in the expression
above. If $|C_{1,k}|=|C_{2,l}|=1$, then 
$mp^{|C_{1,k}\cup C_{2,l}|}=mp^2=o(1)$.
If $|C_{1,k}|=1$ and $|C_{2,l}|>1$, then 
$mp^{|C_{1,k}\cup C_{2,l}|}\leq mp^{|C_{2,l}|}$.
If $|C_{1,k}|>1$ and $|C_{2,l}|=1$, then
$mp^{|C_{1,k}\cup C_{2,l}|}\leq mp^{|C_{1,k}|}$.
If $|C_{1,k}|>1$ and $|C_{2,l}|>1$, then
$
mp^{|C_{1,k}\cup C_{2,l}|}=
mp^{|C_{1,k}|}mp^{|C_{2,l}|}/mp^{|C_{1,k}\cap C_{2,l}|}
$.
Using these bounds results in
\begin{eqnarray*}
\E X(G_\C,\C)&=&
O(1)\prod_{\stackrel{i\in [r]}{|C_{1,i}|>1}}mp^{|C_{1,i}|}
\prod_{\stackrel{i\in [s]}{|C_{2,i}|>1}}mp^{|C_{2,i}|}
\prod_{\stackrel{(k,l)\in \J_3}{|C_{1,k}|>1,|C_{2,l}|>1}}
\frac{1}{mp^{|C_{1,k}\cap C_{2,l}|}}\\
&=&O(1)
\E(X(G_1,\C_1)\E X(G_2,\C_2))
\prod_{\stackrel{(k,l)\in \J_3}{|C_{1,k}|>1,\,|C_{2,l}|>1}}
\frac{1}{mp^{|C_{1,k}\cap C_{2,l}|}}.
\end{eqnarray*}
Therefore, as in the argument for $mp\leq 1$, we have
$$
\prod_{\stackrel{(k,l)\in \J_3}{|C_{1,k}|>1,\,|C_{2,l}|>1}}mp^{|C_{1,k}\cap C_{2,l}|}\ge \prod_{C\in \C_2^\prime[V(G_1)\cap V(G_2)]}mp^{|C|}
\ge \frac{\minS(G_2,\C_2)}{n^\ell},
$$
resulting in (\ref{interm}) for this case.
\end{proof}

\section{Subgraph counts}
In this section we prove Theorem~\ref{main}.
For copies of $\subgraph$ induced by 
clique covers in $\covers_0(\subgraph)$, the proof is
an application of Stein's method, using the estimates
we have already obtained. We must also show that the number
of copies of $\subgraph$ induced by clique covers in 
$\covers(\subgraph)\setminus\covers_0(\subgraph)$
converges to 0 in probability.
\begin{proof}
For each $i\in [N_n]$,
each $\C\in\covers(\subgraph)$ has a corresponding
clique cover in $\covers(H_i)$ induced by the isomorphism
between $\subgraph$ and $H_i$. We write
$X(H_i,\C)$  for the indicator random variable that 
the clique cover in $\covers(H_i)$ corresponding to 
$\C\in\covers(\subgraph)$ induces $H_i$.
Suppose that $\subgraph$ is strictly $\alpha$-balanced
and that $p=cn^{-\eta_0}$ for $c>0$.
Write $X=\sum_{i\in [N_n]}X_i$ as $X=Y_0+Y_1$, where
$$
Y_0=\sum_{i\in[N_n]}\sum_{\C\in \covers_0(\subgraph)} X(H_i,\C),
$$
and
$$
Y_1=\sum_{i\in[N_n]}
\sum_{\C\in \covers(\subgraph)\setminus\covers_0(\subgraph)} 
X(H_i,\C).
$$

For each $\C\in\covers_0(\subgraph)$,
$\eta_0(\subgraph)=\eta_1(\subgraph,\C)=\eta_2(\subgraph,\C,V(\subgraph))$
and by Lemma~\ref{cliqueasymp},
\begin{eqnarray}\nonumber
\BE(Y_0)&\sim&\frac{n^h}{|{\rm aut}(\subgraph)|}
\sum_{\C\in\covers_0(\subgraph)}\pi(\subgraph,\C)\\
\nonumber
&\sim&\lambda_0:=\frac{1}{|{\rm aut}(\subgraph)|} 
\sum_{\C\in\covers_0(\subgraph)}c^{\sum\C}.
\end{eqnarray} 
For any $\C\in\covers_0$ and 
$\emptyset\varsubsetneq S\varsubsetneq V(\subgraph)$,
$\psi(\subgraph,\C,S)\to\infty$ and so $\Phi\to\infty$,
where $\Phi$ is defined by 
\begin{equation}\label{EqPhiEta}
\Phi:=\Phi(\subgraph)=\min_{\C\in\covers_0}(\omega(\subgraph,\C)).
\end{equation}

A {\em dependency graph} $L$ is a graph 
with vertex set $[N_n]\times \covers_0$ having the property that 
whenever $A\subseteq [N_n]\times \covers_0$ and 
$B\subseteq [N_n]\times \covers_0$ satisfy the property that
there are no edges between $A$ and $B$ in $L$,
it follows that $\{X(H_i,\C_1):(i,\C_1)\in A\}$ and $\{X(H_j,\C_2):(j,\C_2)\in B\}$
are mutually independent sets of random variables. We define a dependency graph $L$ with a vertex set $[N_n]\times \covers_0$
and such that for $(i,\C_1),(j,\C_2)\in [N_n]\times \covers_0$, 
$\{(i,\C_1),(j,\C_2)\}\in E(L)$ if and only if
$V(H_i)\cap V(H_j)\neq\emptyset$.
Since subgraphs with disjoint vertex sets appear independently in $\Gp$, this is well defined dependency graph.
We have
$$
\E(X(H_i,\C))=\pi(H_i,\C).
$$
From Theorem 6.23 of \cite{JLR}, we have
\begin{align}\nonumber
d_{TV}(Y_0,P_{\BE(Y_0)})
&\le
\min({\lambda_0}^{-1},1)
\Bigg(\sum_{(i,\C)\in V(L)}\pi(H_i,\C)^2
\\
&+\nonumber
\sum_{(i,\C_1)(j,\C_2)\in E(L)}
\left(\pi(H_i,\C_1)\pi(H_j,\C_2)+\E(X(H_i,\C_1)X(H_j,\C_2))\right)\Bigg),
\end{align}
where the sum $\sum_{(i,\C_1)(j,\C_2)\in E(L)}$ means summing over ordered pairs
$((i,\C_1)(j,\C_2))$ such that $\{(i,\C_1)(j,\C_2)\}\in E(L)$.
Observe that
$$
\sum_{(i,\C)\in V(L)}\pi(H_i,\C)^2
=O\left(\sum_{\C\in\covers_0} n^{h}m^{2|\C|}p^{2\sum\C}\right)=O(n^{-h}).
$$
and that
$$
\sum_{(i,\C_1)(j,\C_2)\in E(L)}
\pi(H_i,\C_1)\pi(H_j,\C_2)=
O\left(n^{2h-1}m^{2|\C|}p^{2\sum\C}\right)=O(n^{-1})
$$
Note that 
$$
\E(X(H_i,\C_1)X(H_i,\C_2))=0
$$
when $\C_1,\C_2\in\covers_0(H_i)$ and $\C_1\neq\C_2$. 
By Lemma~\ref{moment} and \eqref{EqPhiEta}
\begin{align*}
\sum_{(i,\C_1)(j,\C_2)\in E(L)}
&\E(X(H_i,\C_1)X(H_j,\C_2))\\
&=
\sum_{\stackrel
{i,j\in [N_n]}
{V(H_i)\cap V(H_j)\neq \emptyset}
}
\sum_{C_1\in\covers_0(H_i)}
\sum_{C_2\in\covers_0(H_j)}
\E(X(H_i,\C_1)X(H_j,\C_2))\\
&=O(1)
\sum_{C_1,C_2\in\covers_0(\subgraph)}
n^h\sum_{l=1}^{h-1}n^{h-l}
\frac{\pi(\subgraph,\C_1)\pi(\subgraph,\C_2) n^{l}}{\Phi}\\
&=O(\Phi^{-1})=o(1).
\end{align*}
This proves that $d_{TV}(Y_0,P_{\BE(Y_0)})=o(1)$.
Since, as is well known, 
$$d_{TV}(P_{\BE(Y_0)},P_{\lambda_0})=
O\left(|\BE(Y_0)-\lambda_0|\right)=o(1),$$ we have
\eq\label{bound1}
d_{TV}(Y_0,P_{\lambda_0})\leq
d_{TV}(Y_0,P_{\BE(Y_0)})+d_{TV}(P_{\BE(Y_0)},P_{\lambda_0})=o(1).
\en

If $\C\in\covers(\subgraph)\setminus\covers_0(\subgraph)$, then
$\eta_1(\subgraph,\C)<\eta_0(\subgraph)$ and it must be the case
that 
there exists $\emptyset\varsubsetneq S\subseteq
V(\subgraph)$ such that 
$\eta_2(\subgraph,\C,S)<\eta_0(\subgraph)$. 
For this $S\subset V(\subgraph)$,
let $H_i[S]$ be the subgraph of $H_i$ induced by those vertices
of $V(H_i)$ which are the image of $S$ under the isomorphism between $\subgraph$
and $H_i$.
For this $S$ and $p=cn^{-\eta_0}$, by definitions of $\eta_0$, $\eta_1$, and $\eta_2$ we have

$$
\Pra{\sum_{i\in [N_n]}X(H_i,\C)>0}
\leq \Pra{
\exists_{i\in [N_n]} X(H_i[S],\C[S])>0}
\le \psi(\subgraph,\C,S)
=o(1).
$$
Therefore, since there is a finite number of clique covers we have $\Pra{Y_1>0}=o(1)$.
We conclude that 
\begin{eqnarray*}
d_{TV}(X,P_{\lambda_0})&\leq&
d_{TV}(X,Y_0)
+d_{TV}(Y_0,P_{\lambda_0})\\
&\leq&
\Pra{Y_1>0}+d_{TV}(Y_0,P_{\lambda_0})\\
&=&
o(1).
\end{eqnarray*}
\end{proof}

We make some further remarks pertaining to the possibility that
$\BE(X)\not\to\lambda_0$.
We call a clique cover $\C\in\covers_0$
$\alpha$-{\em unbalanced} if
$\eta_2(\subgraph,\C,S)<\eta_1(\subgraph,\C,V(\subgraph))$
for some $\emptyset\varsubsetneq S\varsubsetneq V(\subgraph)$.
If all $\C\in\covers_0$ are strictly $\alpha$-balanced, then 
Theorem~\ref{main} holds and $\lim_{n\to\infty}\BE(X)=\lambda_0$. 
If all of the $\C\in\covers_0$ are either strictly $\alpha$-balanced
or $\alpha$-unbalanced, and at least one of them is strictly $\alpha$-balanced,
then Theorem~\ref{main} holds with $\lim_{n\to\infty}\BE(X)>\lambda_0$. 
If all $\C\in\covers_0$ are $\alpha$-unbalanced, 
then $X\to 0$ in probability when $p=cn^{-\eta_0}$.



\section*{Appendix}

\begin{proof}[Proof of Corollary~\ref{CorTrianglefree}]
Let $\C$ be the sole clique cover of the strictly balanced subgraph
$\subgraph$ consisting of cliques of size two. Then for any $S\subseteq V(\subgraph)$ we have $\sum\C'[S]=2|\C'[S]|=2|E(S)|$ and  $\frac{|S|}{|E(S)|}>\frac{h}{e}$ if $E(S)$ is non-empty.
Recall that $\alpha>\frac{h}{e}$.
\\
Therefore, for any $\emptyset\subsetneq S\subsetneq V(\subgraph)$ such that $E(S)$ is non-empty
\begin{align*}
\frac{|S|+\alpha|\C'[S]|}{\sum \C'[S]}=\frac{|S|}{2|E(S)|}+\frac{\alpha}{2}
>
\frac{h}{2e}+\frac{\alpha}{2}
=\frac{|V(\subgraph)|+\alpha|\C'[V(\subgraph)]|}{\sum \C'[V(\subgraph)]}.
\end{align*}
Now, notice that $|E(S)|=\sum\C[S]-|\C[S]|$. 
Therefore, for any $\emptyset\subsetneq S\subsetneq V(\subgraph)$,
such that $E(S)$ is non-empty and
$\alpha\le\frac{|S|}{\sum \C[S]-|\C[S]|}$, we have $|S|-\alpha |E(S)|\ge 0$. Thus
\begin{align*}
\frac{|S|+\alpha|\C[S]|}{\sum \C[S]}
& =
\frac
{|S|+\alpha(\sum\C[S]- |E(S)|)}
{\sum\C[S]} 
=
\frac{|S|-\alpha|E[S]|}{\sum\C[S]}+\alpha\geq\alpha=\\
&=
\frac{\alpha}{2}+\frac{\alpha}{2}
>
\frac{h}{2e}+\frac{\alpha}{2}
= \frac{|V(\subgraph)|+\alpha|\C'[V(\subgraph)]|}{\sum \C'[V(\subgraph)]}. 
\end{align*}
Finally, in the case $\sum \C[S]-|\C[S]|=0$, i.e.~$E(S)=\emptyset$, we have 
$$
\frac{|S|+\alpha|\C[S]|}{\sum \C[S]}>\alpha>
\frac{h}{2e}+\frac{\alpha}{2}=\frac{|V(\subgraph)|+\alpha|\C'[V(\subgraph)]|}{\sum \C'[V(\subgraph)]}.
$$
Therefore $\subgraph$ is strictly $\alpha$-balanced and Theorem~\ref{main} applies.
\end{proof}

\end{document}